\documentclass[12pt,a4paper]{amsart}
 \pdfoutput=1
 
 \usepackage{amsmath,amssymb,amsthm}
 \usepackage{bm}  
 \usepackage{mathtools}
 \usepackage{thmtools}
 \usepackage{float} 
 
 \usepackage{xcolor} 	
 \usepackage{hyperref}
 \hypersetup{
 	colorlinks,
     linkcolor={red!60!black},
     citecolor={green!60!black},
     urlcolor={blue!60!black},
 }

 \usepackage[utf8]{inputenc}
  \usepackage{csquotes}
 \usepackage[T1]{fontenc}
 \usepackage{lmodern}
 \usepackage[babel]{microtype}
 \usepackage[english]{babel}
 \usepackage[maxbibnames=99]{biblatex}
 \usepackage{geometry}
 \usepackage{enumitem}

 \DeclareMathAlphabet{\mymathbb}{U}{BOONDOX-ds}{m}{n}

\theoremstyle{plain}
\newtheorem{theorem}{Theorem}[section]
\newtheorem{fact}[theorem]{Fact}

\newtheorem{claim}[theorem]{Claim}

\newtheorem{lemma}[theorem]{Lemma}

\newtheorem{observation}[theorem]{Observation}

\newtheorem{question}[theorem]{Question}

\newtheorem*{theorem*}{Theorem}
\newtheorem*{corollary*}{Corollary}

\theoremstyle{definition}

\title{Circuit-partition of infinite matroids}

\author{Nathan Bowler}
\address{Nathan Bowler,
Department of Mathematics, University of Hamburg, Bundesstra{\ss}e 55 (Geomatikum), 20146 Hamburg, Germany}
\email{nathan.bowler@uni-hamburg.de}

\author{Attila Jo\'{o}}
\thanks{Funded by the Deutsche Forschungsgemeinschaft (DFG, German Research Foundation) - Grant No. 513023562 and 
partially by NKFIH 
OTKA-129211}
\address{Attila Jo\'{o},
Department of Mathematics, University of Hamburg, Bundesstra{\ss}e 55 (Geomatikum), 20146 Hamburg, Germany}
\email{attila.joo@uni-hamburg.de}

\begin{document}
\begin{abstract}
 Komjáth, Milner, and Polat investigated when a finitary matroid admits a partition into circuits. They defined the class of 
 ``finite matching 
 extendable'' matroids and showed in their compactness theorem that those matroids always admit such a partition. Their 
 proof is based on Shelah's 
 singular compactness technique and a careful analysis of certain $\triangle$-systems.
 
 We provide a short, simple proof of their theorem. Then we show that a finitary binary oriented matroid can be partitioned 
 into directed circuits if 
 and only if, in every cocircuit, the cardinality of the negative and positive edges is the same. This generalizes an 
 earlier
 conjecture of Thomassen, 
 settled affirmatively by the second author, about partitioning the edges of an infinite directed graph into directed cycles. As 
 side results, a 
 Laviolette theorem for finitary matroids and a Farkas lemma for finitary binary oriented matroids are proven. An example is 
 given to show that, in 
 contrast to finite oriented matroids, `binary' is essential in the latter result.
\end{abstract}
\maketitle

\section{Introduction}
Decomposing complex structures into simple parts is a common theme in mathematical research, extending beyond 
combinatorics. Cycles are among the most fundamental graph classes, leading to a natural question: When is it possible to partition a graph into 
cycles? For 
edge-partitions of 
finite graphs, the answer is straightforward. On the one hand, each cycle contributes an even amount to the degree of each vertex, so the exclusion 
of 
vertices of odd degree is necessary. On the other hand, if all degrees are even, a desired edge-partition can be constructed 
``greedily''. The 
problem becomes significantly harder for infinite graphs, especially uncountable ones. The absence of vertices of odd degree remains necessary but 
is not sufficient. For example, in a two-way infinite path, each vertex has degree two, yet it does not contain a single cycle.
Nash-Williams 
proved in his seminal paper \cite[p. 235 Theorem 3]{nash1960decomposition} that excluding odd cuts\footnote{Finite cuts 
with an odd number of 
edges.} is sufficient (and obviously necessary). New, simpler proofs of his theorem were found by L. Soukup \cite[Theorem 
5.1]{soukup2011elementary} and Thomassen \cite{thomassen2017nash}. In the same paper, Thomassen conjectured that the 
edges of a directed 
graph can be partitioned into directed cycles if and only if, in every cut, the cardinality of edges going in one direction equals the cardinality of 
edges going in the opposite direction. This was confirmed by the second author \cite[Theorem 
2]{joo2021dicycleDecomp}.

Welsh observed \cite{welsh1969euler} that a finite binary matroid\footnote{Definitions and basic facts about finitary 
matroids are given in 
Subsection \ref{subs: matroid}.} can be partitioned into circuits if and only if it has no odd cocircuit, which generalizes the 
analogous observation 
about graphs. Polat extended this result to finitary binary matroids \cite[Theorem 2.1 (iv) and (vi)]{polat1987compactness}, 
providing a common 
generalization of Nash-Williams' theorem and Welsh's result. Later, together with Komjáth and Milner \cite[Theorem 
1]{komjath1988compactness}, he further generalized his result by proving that matroids, they call ``finite matching 
extendable'' (f.m.e.), can always 
be partitioned into circuits (they call matroids with this property ``matchable''):
\begin{restatable*}[{Komjáth, Milner, Polat; 
\cite[Theorem1]{komjath1988compactness}}]{theorem}{fme}\label{thm: fme match}
Every finitary f.m.e. matroid is matchable.
\end{restatable*}

 Since binary matroids that do not contain odd cocircuits are f.m.e., this is indeed a generalization.  
Their proof is based on Shelah's singular compactness 
technique and a careful analysis of certain $\triangle$-systems. We begin by providing a short and relatively simple 
proof of their 
theorem (Theorem \ref{thm: fme match}).

A natural question is whether it is also possible to generalize the directed cycle edge-partition result \cite[Theorem 
2]{joo2021dicycleDecomp} to 
finitary binary oriented matroids, as Polat generalized Nash-Williams' theorem to finitary binary matroids. The first difficulty 
corresponds to the 
so-called Farkas lemma. In a directed graph, it is easy to see that every directed edge is contained in a directed cycle or a 
directed cut. This is 
known to be true for finite oriented matroids, but we will show that it may fail severely in infinite ones (see Claim \ref{claim: 
Farkas fail}). However, 
we prove that it holds true in the binary case:
\begin{restatable*}[Farkas lemma for finitary binary oriented matroids]{lemma}{Farinf}\label{lem: Farkas 
property}
Let $ \vec{M}=(E, \mathcal{O}) $ be a finitary binary oriented matroid. Then for every $ e\in E $, there is either a positive 
circuit 
or a positive cocircuit through $ e $ but not both.
\end{restatable*}

Using this, we prove that a finitary 
binary oriented matroid 
can be partitioned into directed circuits if and only if, in every signed cocircuit, the cardinality of the negative and positive 
edges is 
the same (in which case we refer to that oriented matroid as ``balanced''):
\begin{restatable*}{theorem}{orparti}\label{thm: poscirc partition}
A finitary binary oriented matroid can be partitioned into directed circuits if and only if it is balanced.
\end{restatable*}

Let $\lambda$ be an infinite cardinal. An edge-partition $\{ G_i: i \in I \}$ of a graph $G$ into subgraphs of size at most 
$\lambda$ is called 
$\lambda$-bond faithful if in each $G_i$ every bond of size less than $\lambda$ is also a bond of $G$, and each bond of $G$ 
of size at most 
$\lambda$ is a bond of a suitable $G_i$. Laviolette proved that an $\aleph_0$-bond faithful partition always exists, and, 
under the Generalized 
Continuum Hypothesis, $\lambda$-bond faithful partitions exist for every $\lambda$ (see \cite[Proposition 
3]{laviolette2005decompositions}). L. 
Soukup showed that the Generalized Continuum Hypothesis can be omitted (see \cite[Theorem 
6.2]{soukup2011elementary}).  We define $ \lambda $-cocircuit faithful 
partitions of matroids analogously to $\lambda$-bond faithful partitions,  and prove the following  generalization of 
Laviolette's theorem:
\begin{restatable*}[Laviolette theorem for finitary matroids]{theorem}{Laviolette}\label{thm: finitary 
matroid Lavio}
For every infinite cardinal $\lambda $, every finitary matroid $ M=(E, \mathcal{C}) $ admits a $ \lambda $-cocircuit faithful 
partition. 
\end{restatable*}

This paper is organized as follows. In the next section, we survey the basic facts and notation related to set theory and matroid 
theory that we will 
use. Sections \ref{sec: unoriented} and \ref{sec: Lavio}, due to the second author, discuss the new short proof of the 
compactness theorem by 
Komjáth et al. and the generalization of Laviolette's theorem. Sections \ref{sec: farkas} and \ref{sec: oriented} are based on 
previous works of the 
first author with Carmesin, who kindly gave us permission to use them. Their respective contents are the generalization of the 
Farkas lemma 
and the directed circuit partition result. The framework of our proofs involves cutting up uncountable matroids by a chain of 
elementary 
submodels, a method popularized by L. Soukup in \cite{soukup2011elementary}. Finally, we raise a 
couple of open questions in Section \ref{sec: open questions}.

\section{Notation and preliminaries}\label{sec: prelim}
\subsection{Set theory}\label{subs: set theory}
We use variables $ \alpha $ and $ \beta $ to stand for ordinal numbers, while $ \kappa $, $ \lambda $ and $ \Theta $ denote 
cardinals. An infinite cardinal $\kappa$ is called \emph{regular} if it cannot be expressed as the sum of fewer than $\kappa$ many 
cardinals, each of which is less than $\kappa$.   We write $ \bigcup X $  for the union of the sets in $ X $. A function $ f $ 
is 
a set of 
ordered pairs such that if $ \left\langle x, y  \right\rangle, \left\langle x,z  \right\rangle\in f $, then $ y=z $. If $ X $ is any set, 
(not necessarily $ X\subseteq \mathsf{dom}(f) $), then 
we write $ f\!\!\upharpoonright\!\! X $ for $ \{ \left\langle x, y  \right\rangle\in f:\ x\in X \} $. 

 \subsubsection{Compactness principle}
 The compactness principle in infinite combinatorics allows us to use finite reasoning to prove infinite results. A little more precisely, suppose that 
 we want to choose suitable values for each of infinitely many variables, each of which can take any one of finitely many possible values. If we must 
 choose these values subject to an 
 infinite set of
 constraints, each of 
 which depends only on finitely many variables, then the compactness principle says that we can set the variables to 
 satisfy all the constraints if we can do so for every finite subset of the constraints.   The first historical application of the compactness principle
  in infinite combinatorics is a short proof of a famous De Bruijn–Erdős theorem \cite{bruijn1951colour} due to Gottschalk 
  (see in \cite{gottschalk1951choice}). The theorem states
  that if  every finite subgraph of an infinite graph $ G $ is $ k $-colourable\footnote{A graph is $ k $-colourable if its vertices 
  can be coloured by at most $ k $ colours in such a way that the endpoints of each edge are assigned different colours.} for 
  a $ k<\omega $, then so is $ G $. We give here a brief demonstration of the compactness principle by repeating 
  Gottschalk's proof.
 
  Recall that a topological 
 space $ (X, \tau) $ is called \emph{compact} 
 if whenever $ \mathcal{F} $ is a set of closed subsets of $ X $ such that every finite subfamily of $ \mathcal{F} $ has 
 nonempty intersection, then so does the whole family $ \mathcal{F} $. Consider the compact space consisting of $ \{ 
 1,\dots, k \} $ equipped with the 
 discrete topology. For $ v\in V(G) $,  let $ (X_v, \tau_v) $ be a copy of this space and let $ (X, \tau):=\prod_{v\in 
 V(G)}(X_v, \tau_v) $. Then $ (X, \tau) $ is compact by Tychonoff's theorem. For $ uv\in E(G) $, let $ C_{uv}\subseteq X $ 
 consists 
 of those $ x $ for which $ x(u)\neq x(v) $. Then $ 
 \mathcal{F}:= \{ C_{e}:\  e\in E(G) \} $ is a  family of closed sets. In order to show that $ G $ is $ k $-colourable, we need 
 to prove that the 
 intersection of all the sets in $ \mathcal{F} $ is nonempty.   The intersection of any finitely many sets in $ 
\mathcal{F} $ is nonempty because the finite subgraphs of $ G $ are properly $ k $-colourable by assumption. But then by 
the compactness of the space $ (X, \tau) $, the intersection of all sets in $ \mathcal{F} $ is also nonempty.

In this example the variables to be chosen are the colours of the vertices of $G$, and the closure of each $C_e$ follows from the fact that the 
constraint it encodes only mentions finitely many (in this case just two) of these variables.

For a more detailed introduction to the compactness principle in combinatorics, with more examples, see \cite{diestel2017graph}.

\subsubsection{Elementary submodels}
Dividing an uncountable structure into smaller pieces in order to use an inductive argument is a common strategy in 
proofs about uncountable structures. In order to make the induction hypothesis applicable, one needs to make sure that the pieces 
inherit certain properties that are relevant with respect to the problem. The desired properties would typically state that 
these pieces are closed under certain operations. Constructing such closed pieces ``by hand'' takes a lot of effort and often requires a long 
description of various operations, tailored to the problem, to close under. 

This is where a tool from Model Theory comes in useful, namely elementary submodels. These are a kind of substructure which are, informally 
speaking, closed under every reasonable operation you could think of. Furthermore, there are standard theorems which allow us to extend arbitrary 
substructures to elementary ones without increasing their cardinality. Thus the theory of elementary submodels provides a simple, efficient way to 
make 
the pieces closed under all possible operations without explicitly describing any of those operations. 

A typical approach is to pick a set that is big enough to be a ``good 
approximation''  of the whole set-theoretic universe\footnote{Unfortunately, for technical reasons we cannot simply use the whole set-theoretic 
universe directly, since it is a proper class.} and contains all the objects that are relevant for the proof, then construct a chain of elementary 
submodels of this set. Then this chain usually partitions the uncountable structure in question into pieces that inherit some important properties of 
the whole structure. For a gentle introduction to elementary submodels and their combinatorial applications we refer the reader to 
   \cite{soukup2011elementary}. Next we recall some important facts about elementary submodels, which will allow us to flesh out the above ideas 
   in more detail.

The transitive closure of a set $ 
X $  is the set whose elements are: the elements of $ X $, the elements of the elements of $ 
  X $ etc. For a cardinal $ \varTheta $,  $ H(\varTheta) $ denotes the set of those sets $ X $ whose transitive closure  is 
 smaller than $ \Theta $.  We write  $ 
 \mathcal{E}\prec H(\varTheta) $ if  $ (\mathcal{E}, \in) $ is an elementary 
 submodel of the first order structure $ (H(\Theta),\in) $, i.e.  $\emptyset\neq  \mathcal{E} \subseteq H(\varTheta) $ and for 
 every 
 formula $ 
 \varphi(v_0,\dots, v_{n-1}) $ in the language of set theory and for every $ x_0,\dots, x_{n-1}\in \mathcal{E} $: $ 
 \varphi(x_0,\dots, x_{n-1}) $ holds true in the structure $ (\mathcal{E}, \in) $ if and only if it holds true in $ (H(\Theta),\in) 
 $.  
 \begin{fact}[{\cite[Claim 3.2]{soukup2011elementary}}]\label{fact: finite subsets}
 If $ \varTheta $ is an infinite cardinal and $ \mathcal{E} \prec H(\varTheta) $, then $ \mathcal{E} $ contains all of its own finite subsets.
 \end{fact}
 \begin{fact}[{implicit in \cite[Corollary 2.6]{soukup2011elementary}}]\label{fact: elementary chain}
Let $ \aleph_0\leq \lambda<\kappa<\varTheta $ be cardinals where $ \varTheta $ is regular and let $ x\in H(\varTheta) $. 
Then there is a 
$ \subseteq $-increasing 
continuous sequence $ \left\langle \mathcal{E}_\alpha:\ 
 \alpha<\kappa  \right\rangle $ where $\mathcal{E}_0=\emptyset $ and  $ \mathcal{E}_\alpha\prec H(\varTheta) $ for $ 1\leq 
 \alpha<\kappa $ 
  such that $ x\in \mathcal{E}_1 $ 
 and for every $ \alpha<\kappa $:  $ 
 \left|\mathcal{E}_{\alpha}\right| \cup\{ \mathcal{E}_\alpha  \}
 \subseteq\mathcal{E}_{\alpha+1} $  and $ \left|\mathcal{E}_\alpha\right| =\left|\alpha\right|\cdot \lambda$. 
 \end{fact}
 \begin{fact}[{\cite[Claim 3.7]{soukup2011elementary}}]\label{fact: big set}
  Let $ \varTheta $ be an uncountable cardinal and let $ \mathcal{E}\prec H(\varTheta) $ with $ 
  \left|\mathcal{E}\right|\subseteq 
  \mathcal{E} $. If $ X\in \mathcal{E} $ with $ X\setminus \mathcal{E} \neq \emptyset  $, then $ \left|X\right|> 
  \left|\mathcal{E}\right| $ and $ \left|X\cap \mathcal{E} \right|= \left|\mathcal{E}\right|  $.
  \end{fact}
 
 \begin{fact}\label{fact: lot of disjoint}
 Let $ \varTheta $ be an uncountable cardinal,  $ \mathcal{E}\prec H(\varTheta) $ with $ 
 \left|\mathcal{E}\right|\subseteq 
 \mathcal{E} $ and let $ \mathcal{X}\in \mathcal{E} $ be a set of countable sets.  If there is a nonempty $ X\in 
 \mathcal{X} $ with 
 $ X \cap \mathcal{E}=\emptyset  $, then  for any maximal set $ \mathcal{Y}\in \mathcal{E} $ of pairwise disjoint elements 
 of $ \mathcal{X} $ 
 we have  
 $  \left|\mathcal{Y}\right|>\left|\mathcal{E}\right|$ and $ \left|\mathcal{Y}\cap 
  \mathcal{E}\right|=\left|\mathcal{E}\right| $. 
 \end{fact}
 \begin{proof}
 If $ 
  \mathcal{Y}\subseteq \mathcal{E} $, then $ \bigcup \mathcal{Y}\subseteq \mathcal{E} $ and $ \mathcal{Y}\cup \{ X \} $ 
  contradicts the 
  maximality of $ \mathcal{Y} $. Therefore $ \mathcal{Y}\setminus \mathcal{E}\neq \emptyset $ and we are done by 
  applying Fact \ref{fact: big 
  set}.
 \end{proof}

\subsection{Matroid theory}\label{subs: matroid}
Matroids were introduced by  Whitney and Nakasawa independently (see \cite{Whitney_1935}) to 
generalize the concept of linear independence in vector 
spaces. In this paper we only work with the most basic concept of infinite matroids, namely finitary matroids. A more 
general concept of infinite matroids was discovered by Higgs \cite{higgs1969matroids} and rediscovered independently by 
Bruhn et al. \cite{bruhn2013axioms}. For a survey about basic facts about infinite matroids we refer to \cite{nathanhabil}.

An ordered pair $ M=(E,\mathcal{C}) $ is a \emph{finitary matroid} (defined by its circuits) if $ \mathcal{C} $ is a set of 
nonempty, pairwise $ 
\subseteq $-incomparable finite subsets of $ E $ satisfying the following \emph{circuit elimination axiom}: If $ C_0, 
C_1\in 
\mathcal{C} $ are distinct and 
$ e\in C_0 \cap C_1 $, then there is a $ C \in \mathcal{C} $ with $ C \subseteq (C_0\cup C_1)\setminus \{ e \} $.
The elements of $ \mathcal{C} $ are called the \emph{circuits} of $ M $. 
\begin{fact}[Strong circuit elimination]\label{fact: strong elim}
If $ M=(E,\mathcal{C}) $ is a finitary matroid, $ C_0, 
C_1\in 
\mathcal{C} $ with $ e\in C_0\setminus C_1 $ and 
$ f\in C_0 \cap C_1 $, then there is a $ C \in \mathcal{C} $ with $e\in  C \subseteq (C_0\cup C_1)\setminus \{ f \} $.
\end{fact} 
Let  $ M=(E, \mathcal{C}) $ be a finitary matroid. Let $ F $ be any set\footnote{In order to simplify some 
later notation, we do not even assume $ F\subseteq E $.}. We 
write $ 
M\setminus F $ for the matroid with  $E(M\setminus F):= E\setminus F $ and $ \mathcal{C}(M\setminus F):=\{ C\in 
\mathcal{C}:\ 
C\cap F=\emptyset \} $. 
\begin{fact}\label{fact: contract circuit}
For every set $ F $, the minimal elements of the set $ \{C\setminus F:\ C\in \mathcal{C} \}\setminus \{ \emptyset \} $ form the set of circuits of a 
matroid.
\end{fact}
\noindent We write $ M/F $ for the matroid on $ E\setminus F $ where $ \mathcal{C}(M/F) $ is the set described in Fact 
\ref{fact: contract circuit}. Let us denote $ M\setminus (E\setminus F) $ and $ M/ (E\setminus F)$ by $ M 
| F $ and $ M.F$ respectively.  Matroids of the form $ M/F \setminus G $ are called the 
\emph{minors} of $ M $.
The set $ 
 \mathcal{C}^{*} $ of \emph{cocircuits} of  $ 
 M $ consists of the minimal nonempty subsets $ D $ of $ E $ for which $ \left|D\cap C\right|\neq 1 $ for every $ C\in 
 \mathcal{C} $.
  \begin{fact}\label{fact: circ cocirc meets 2}
 For every $ D\in \mathcal{C}^{*} $ and for every distinct $ e,f \in D $, there is a $ C\in \mathcal{C} $ with $ C\cap D=\{ 
 e,f \} $.
  \end{fact}
 
 \begin{fact}\label{fact: contract cocircuit}
The set $ \mathcal{C}^{*}(M\setminus F) $ consists of the minimal elements of $ \{D\setminus F:\ D\in \mathcal{C}^{*} 
\}\setminus \{ 
\emptyset 
 \} $ while $ \mathcal{C}^{*}(M/F)= \{ D\in 
 \mathcal{C}^{*}:\ 
 D\cap F=\emptyset \}$.
 \end{fact}
\begin{fact}\label{fact: football field}
 For every $e\in E$ and every partition $X\cup Y$ of $E\setminus \left\lbrace e\right\rbrace $, there is either a $C\in \mathcal{C}$ with $e\in 
 C\subseteq 
 X\cup \left\lbrace e \right\rbrace $ or a $D\in \mathcal{C}^{*}$ with $e\in D\subseteq Y\cup \left\lbrace e \right\rbrace $.
  \end{fact}

If $ G=(V,E) $ is a graph, then the edge 
 sets of the cycles of $ G 
 $ are the circuits of a finitary matroid on $ E $ which is 
  called the \emph{cycle matroid} of $ G $. Its cocircuits are the bonds of $ G $. A vector system with the minimal linearly 
  dependent subsets as 
  circuits always forms a finitary 
  matroid.  Such matroids (as 
well as matroids that 
 are isomorphic to such a matroid) are called \emph{linear} or 
 \emph{representable}.  If the vector system 
 in question is over $ \mathbb{F}_2$, then it is a \emph{finitary binary matroid}. The cycle matroid of a graph is a binary 
 matroid.

 \begin{fact}[{\cite[Theorem 2.3]{duchamp1989characterizations}}]\label{fact: binary char}
 A finitary matroid $ M$ is binary if and only if $ \left|C \cap D\right| $ is even for every $ C\in 
  \mathcal{C} $ and $ D\in 
  \mathcal{C}^{*} $.
 \end{fact}
 
By the nature of linear independence, representable matroids are always finitary. Representability should not 
be
confused with the more general concept of thin sum representability introduced in \cite{borujeni2015thin}, which does not 
require that the matroid be finitary.

 \subsection{Oriented matroids}\label{subs: oriented matroids}
 Oriented binary matroids were defined first by Minty in \cite{minty1966axiomatic}. These are also 
 known as signed matroids (see \cite[p. 458]{oxley1992book}). The more 
 general 
 concept of oriented 
 matroids was discovered  independently by Bland and Las Vergnas \cite{bland1978orientability} and by Folkman and 
 Lawrence \cite{folkman1978oriented}. General concepts of infinite oriented matroids were introduced by 
 Hochst\"attler and  Kaspar 
 \cite{hoch2015orthogax}.  Infinite 
 binary orientable matroids are characterised (under the name of signable matroids) by Bowler and 
 Carmesin in \cite[p. 111]{bowler2018excluded}. Affine representable infinite oriented matroids were investigated recently 
 by 
 Delucchi and Knauer \cite{delucchi2024finitary}.
 
 A \emph{signing} of a set is a $ \pm 
1$-valued function defined on the set. A \emph{signed set} is a function which is a signing of some set.   A 
signed set is \emph{positive} if it is constant $ 1 $ and \emph{negative} if it is 
 constant $ -1 $.  
 The domain of a signed set $ X $ is 
 denoted by $ \underline{X} $.  We define $ X^{+}:= \{ e\in \underline{X}:\ X(e)=1 \} $ and  $ X^{-}:= \{ e\in 
 \underline{X}:\ X(e)=-1 \} $. If $ 
 \mathcal{X} $ is a set of signed sets, then let $ \underline{\mathcal{X}}:=\{ 
 \underline{X}:\ X\in \mathcal{X} \} $.
 A 
 \emph{signature} of 
 a set $ \mathcal{F} $ of nonempty sets is a set that consists of exactly two signings of each  $ F \in \mathcal{F} $ and these 
 two 
 are opposite. An \emph{orientation} $ \vec{M} $ of a finitary matroid $ M=(E, \mathcal{C}) $ is a pair $ (E,\mathcal{O}) $ 
 where $ \mathcal{O} $ is a signature of  $ \mathcal{C} $ that satisfies the following \emph{oriented circuit elimination 
 axiom}: For every $ C_0, C_1 \in \mathcal{O} $ with $ C_1 \neq \pm C_0 $ and $ e\in C_0^{+}\cap C_1^{-} $,  
  there is a $ C\in \mathcal{C} $ 
  with $ C^{-}\subseteq (C_0^{-}\cup C_1^{-})\setminus \{ e \} $ and $ C^{+}\subseteq (C_0^{+}\cup C_1^{+})\setminus 
  \{ e \} $.  Every oriented matroid gives rise to a unique signature $ \mathcal{O}^{*} $ of its cocircuits satisfying
  
  \begin{equation}\label{ax: orthog}
  (\forall C\in \mathcal{O})(\forall D \in \mathcal{O}^{*}) \left[\underline{C} \cap 
    \underline{D}\neq \emptyset  \Longrightarrow C \!\!\upharpoonright\!\! (\underline{C} \cap 
  \underline{D})\neq \pm D \!\!\upharpoonright\!\! (\underline{C}\cap \underline{D})\right] . 
  \end{equation}
 
 Conversely, if $ \mathcal{O}$ and $ \mathcal{O}^{*} $ are signatures 
of the circuits and of the cocircuits respectively and they satisfy \emph{orthogonality axiom}  (\ref{ax: orthog}), then $ 
\mathcal{O} 
$ satisfies the oriented circuit elimination 
axiom. The orientations of finitary matroids are called oriented finitary matroids.

\begin{fact}\label{fact: bin orthog}
A finitary oriented matroid $\vec{M}=(E,\mathcal{O}) $  is binary if and only if  $ \sum_{e\in 
\underline{C}\cap 
\underline{D}}C(e)D(e)=0 $ for every $ C\in \mathcal{O} $ and $ D \in \mathcal{O}^{*} $.
\end{fact}
\noindent The fact above is well-known in the finite case \cite[Corollary 13.4.6]{oxley1992book} and remains true by 
compactness for 
finitary matroids.

We refer to the positive and negative (co)circuits of $ \vec{M} $ as its \emph{directed (co)circuits}.
\begin{fact}[{Farkas lemma \cite[Corollary 3.4.6]{bjorner1999oriented}}]\label{fact: Farkas lemma}
If $ \vec{M}=(E,\mathcal{O}) $ is a finite oriented matroid, then for every $ e\in E$ there is either a positive $ C\in 
\mathcal{O} $ through $ e $ or a positive $ D\in \mathcal{O}^{*} $ through $ e $ but not both.
\end{fact} 
Minors of finitary oriented matroids are defined similarly to the unoriented ones, except they inherit the signs as well. 

\begin{fact}[{\cite[Lemma 3.6]{hoch2015orthogax}}]\label{fact: oriented minor}
Let $ \vec{M}=(E, \mathcal{O}) $ be a finitary oriented matroid and let $ F $ be a set. Then 
\begin{align*}
\mathcal{O}(\vec{M}\setminus F)&=\{ C\in \mathcal{O}:\ \underline{C}\cap F=\emptyset \}, \\
\mathcal{O}(\vec{M}/F)&=\{ C \!\!\upharpoonright\!\! (E\setminus F):\ C\in \mathcal{O},\  
\underline{C}\setminus 
F \in 
\mathcal{C}(M/F)\},                 \\
\mathcal{O}^{*}(\vec{M}\setminus F)&=\{ D \!\!\upharpoonright\!\! (E\setminus F):\ D\in \mathcal{O}^{*},\ 
\underline{D}\setminus F \in 
\mathcal{C}^{*}(M/F)\},      \\
\mathcal{O}^{*}(\vec{M}/ F)&=\{ D\in \mathcal{O}^{*}:\ \underline{D}\cap F=\emptyset \}.
\end{align*}
\end{fact}

A finitary matroid is called \emph{regular} if it is representable over every field. Finitary binary orientable 
matroids are 
exactly the finitary regular matroids. By replacing representability with the already mentioned thin sum representability, this 
connection extends 
even beyond finitary matroids (see \cite[Subsection 5.1]{bowler2018excluded}).

\section{Finite matching extendable matroids}\label{sec: unoriented}
 Using the terminology of Komjáth et al. from
 \cite{komjath1988compactness}, a
 \emph{matching} in a matroid $ M=(E, \mathcal{C}) $ is a set $ 
 \mathcal{A} $ of pairwise disjoint elements of $ \mathcal{C} $. A \emph{perfect matching} is a matching $ \mathcal{A} $ 
 with $ 
 \bigcup \mathcal{A}=E $, i. e. a partition of the ground set of $ M $ into circuits. A matroid is \emph{matchable} if it 
 admits 
 a perfect matching. A matroid $ M $ is called  \emph{finite 
 matching 
 extendable} 
 (abbreviated as f.m.e.) if for every finite matching $ \mathcal{A} $ and for every edge 
 $ e \in E\setminus \bigcup \mathcal{A} $ there is a circuit $ C $ with $ e\in C \subseteq E\setminus \bigcup \mathcal{A} $. 
 
\fme 

\begin{proof}
Let $ \varTheta $ be an arbitrary but fixed regular uncountable cardinal. It is sufficient to prove the theorem only for matroids 
in $ H(\varTheta) $.

\begin{observation}\label{obs: down OK}
If $ \mathcal{E}\prec H(\varTheta) $, then for every finitary f.m.e. matroid  $ M\in \mathcal{E} $, the matroid $ 
M| \mathcal{E} 
$ is also f.m.e. 
\end{observation}
\begin{proof}
Let  $ \mathcal{A} $ be a finite matching in $ M| \mathcal{E} $, and an let $ e\in E\cap \mathcal{E} 
\setminus \bigcup 
\mathcal{A} $. Then for each $ C\in \mathcal{A} $ we have $ C\in \mathcal{E} $, moreover, $ \mathcal{A} \in 
\mathcal{E} $ (see Fact \ref{fact: finite subsets}).   Thus $ M, 
\mathcal{A},e \in \mathcal{E} $ and hence there is a $ C\in \mathcal{C}\cap \mathcal{E} $ with $ e \in C \subseteq 
E\setminus \bigcup \mathcal{A} 
$. Finally, $ C\subseteq \mathcal{E} $ because $ C $ is finite (see Fact \ref{fact: big set}), therefore $ C \in 
\mathcal{C}(M| 
\mathcal{E}) $.
\end{proof}

\begin{lemma}\label{lem: key lemma}
If $ \mathcal{E}\prec H(\varTheta) $ with $ \left|\mathcal{E}\right|\subseteq 
\mathcal{E} $, then for every  finitary f.m.e. matroid $ M\in \mathcal{E} $,  the matroid $ M\setminus 
\mathcal{E}$ is also f.m.e.
\end{lemma}
\begin{proof}
Let a finite matching $ \mathcal{A} $ in $ M\setminus \mathcal{E} $ and an edge $ e\in E\setminus (\bigcup 
\mathcal{A}\cup\mathcal{E})  $ be 
given. By applying the assumption that $ M $ is f.m.e., we pick a $ C_0\in \mathcal{C} $ with $ e\in C_0 \subseteq 
E\setminus \bigcup 
\mathcal{A} $ for 
which
$F:=C_0\cap\mathcal{E} $ is minimal. If $ F=\emptyset $, then $ C_0 \subseteq E\setminus \mathcal{E} $ and hence $ C_0 
\in 
\mathcal{C}(M\setminus \mathcal{E}) $, 
thus we are done. Suppose for a contradiction that $F\neq\emptyset $. Let $ \mathcal{X} $ consist of the finite sets that 
extend $ 
F $ to an $ M $-circuit, formally $ \mathcal{X}:=\{X\subseteq E\setminus F:\ F\cup X\in \mathcal{C}  \} $.  Since $ 
C_0\setminus \mathcal{E}\in 
\mathcal{X}$ is clearly disjoint from $ \mathcal{E} $, we conclude by applying Fact 
\ref{fact: lot of disjoint} that there is a set $ \mathcal{Y} $ of pairwise disjoint elements of $ \mathcal{X} $ with $ 
\left|\mathcal{Y}\right|>\left|\mathcal{E}\right| $. Since 
$ \left|\mathcal{E} \cup \bigcup\mathcal{A}\cup \{ e \}\right|=\left|\mathcal{E}\right| $,  it follows that there is an $ X\in 
\mathcal{Y} $ that is 
disjoint 
from $ 
\mathcal{E} \cup \bigcup\mathcal{A}\cup \{ e \} $. Let $ C_1:=F \cup X $. Since $ F\neq \emptyset $ by the indirect 
assumption, we can pick $ f 
\in C_0 \cap C_1 $. By strong circuit elimination (Fact \ref{fact: strong elim}), it follows that there is a  $ C\in \mathcal{C} $ 
with $ e\in C \subseteq 
(C_0 
\cup C_1)\setminus \{ f \} $. But then $ e\in 
C \subseteq E\setminus \bigcup \mathcal{A} $ with $ \left|C\cap \mathcal{E}\right|\leq \left|C_0 \cap \mathcal{E}\right|-1 $ 
(because $C\cap 
\mathcal{E}\subseteq C_0 \cap \mathcal{E} \setminus \{ f \} $) contradicting the 
minimality of $ C_0 $.
\end{proof}

Let $ M=(E, \mathcal{C})\in H(\varTheta) $ be a f.m.e. matroid. We apply transfinite induction on $ \kappa:=\left|E\right| $.
If $ \kappa 
\leq \aleph_0 $,  the statement follows by a 
straightforward recursion. Indeed,  we fix a $\kappa$-type well-order of $ E $  and in step $ n $ apply finite matching 
extendability to 
pick a circuit $ C_n $ through 
the least element of $ E\setminus \bigcup_{i<n}C_i $ that is disjoint from $ C_0,\dots, C_{n-1} $.  

Suppose $ \kappa >\aleph_0 $. Let $ \left\langle \mathcal{E}_\alpha:\ \alpha<\kappa  \right\rangle $ be as in Fact \ref{fact: 
elementary chain} with 
$ 
x:=M $ and $ \lambda:=\aleph_0 $.
By Lemma \ref{lem: key lemma} we know that $ M\setminus \mathcal{E}_\alpha $ is f.m.e. for every $ \alpha<\kappa $. 
Since $ 
M,\mathcal{E}_\alpha \in \mathcal{E}_{\alpha+1} $, we have $ M\setminus \mathcal{E}_\alpha \in 
\mathcal{E}_{\alpha+1} 
$. But then applying Observation \ref{obs: down OK}  with $  M\setminus \mathcal{E}_\alpha $ and $ 
\mathcal{E}_{\alpha+1} $ for $ \alpha<\kappa $, we 
conclude that  $  (M\setminus \mathcal{E}_\alpha) 
| 
\mathcal{E}_{\alpha+1}=M| 
(\mathcal{E}_{\alpha+1}\setminus 
\mathcal{E}_{\alpha}) 
$  is f.m.e. for every $ \alpha<\kappa $. By the induction hypothesis we can pick a perfect matching $ \mathcal{A}_\alpha $ 
of $ 
M | (\mathcal{E}_{\alpha+1}\setminus 
\mathcal{E}_{\alpha}) $ for each $ \alpha<\kappa $. Clearly, $\bigcup_{\alpha<\kappa}\mathcal{A}_\alpha $ is a matching 
in $ 
M $ and it is perfect 
because  $ E= \bigcup_{\alpha<\kappa} E\cap (\mathcal{E}_{\alpha+1}\setminus \mathcal{E}_\alpha) $ is a partition.
\end{proof}

\section{A Laviolette theorem for finitary matroids}\label{sec: Lavio}
Let $ M=(E, \mathcal{C}) $ be a finitary matroid and let $ \lambda $ be an infinite cardinal. An $ F\subseteq E $ is $ 
\lambda 
$\emph{-cocircuit faithful} with respect to $ M $ if  for every $ 
D\in \mathcal{C}^{*}(M| F) $ with $ \left|D\right|<\lambda $ we have $ D\in 
\mathcal{C}^{*}(M) $. A $ \lambda $\emph{-cocircuit faithful partition} of $ M $ is a partition $ \{ F_i:\ i\in I \} $ of $ E $  
such that
\begin{enumerate}[label=(\roman*)]
\item\label{item: small} $ \left|F_i\right|\leq \lambda $ for every $ i\in I $;
\item\label{item: faithful} $ F_i $ is $ \lambda $-cocircuit faithful for each $ i\in I $;
\item\label{item: is somewhere} For every $ D\in \mathcal{C}^{*} $ with $ \left|D\right|\leq \lambda $, there exists an $ i\in 
I  $ with $ D 
\subseteq 
F_i $.
\end{enumerate}
\Laviolette

\begin{proof}
Let $ \varTheta $ be an arbitrary but fixed regular uncountable cardinal. It is sufficient to prove the theorem only for matroids 
in $ H(\varTheta) $.
\begin{lemma}\label{lem: Laviolette key}
Let $ \mathcal{E}\prec H(\varTheta) $ with $ \left|\mathcal{E}\right|\subseteq 
\mathcal{E} $. Then for every finitary matroid $ M=(E,\mathcal{C})\in \mathcal{E} $ and for every $D\in \mathcal{C}^{*} $ with $ D \cap 
\mathcal{E}\neq 
\emptyset $ we have: $ D \subseteq \mathcal{E} $ if $ \left|D\right| \leq \left|\mathcal{E}\right| $ and $ 
\left|D\cap\mathcal{E}\right|=\left|\mathcal{E}\right| $ 
if $ \left|D\right| > \left|\mathcal{E}\right| $.
\end{lemma}
\begin{proof}
We may assume that $ D\setminus \mathcal{E}\neq \emptyset $ since if $ D\subseteq \mathcal{E} $, then the statement holds trivially. We 
need to show that $ 
\left|D\right| > \left|\mathcal{E}\right| $ and 
$\left|D\cap\mathcal{E}\right|=\left|\mathcal{E}\right| $. Let $ e\in D\cap \mathcal{E} $ and $ f\in D\setminus \mathcal{E} $. 
By Fact \ref{fact: 
circ cocirc meets 2}, there is 
a $ C\in \mathcal{C} $ 
with $ 
C \cap D=\{ e,f \} $. Let $ F:=C\cap \mathcal{E} $ and $ \mathcal{X}:=\{X\subseteq E\setminus F:\ F\cup X\in \mathcal{C}  \} $. Note that each $ 
X\in \mathcal{X} $ 
meets $ D $ since otherwise $ (F\cup X)\cap D=\{ e \} $ contradicts $ D\in \mathcal{C}^{*} $. Clearly, $ C \setminus 
\mathcal{E}\in \mathcal{X} $ is disjoint from $ \mathcal{E} $. Let $ \mathcal{Y}\in \mathcal{E} $ be a maximal family of 
pairwise disjoint 
elements of $ 
\mathcal{X} $.  By Fact \ref{fact: lot of disjoint} we know that $ \left|\mathcal{Y}\right|>\left|\mathcal{E}\right| $ and $ 
\left|\mathcal{Y}\cap 
\mathcal{E}\right|=\left|\mathcal{E}\right| $. Since the elements of $ \mathcal{Y} $ are pairwise disjoint and each of them 
meets $ D $,  we 
conclude that $ 
\left|D\right|>\left|\mathcal{E}\right| $ and $ \left|D\cap\mathcal{E}\right|=\left|\mathcal{E}\right|$.
\end{proof}

Let $ M=(E, \mathcal{C})\in H(\varTheta) $ be a finitary matroid. We apply transfinite induction on $ \kappa:=\left|E\right| $. 
If $\kappa\leq 
\lambda $, then the singleton partition $ \left\lbrace E \right\rbrace  $ is $\lambda$-cocircuit faithful. If $\kappa>\lambda $, then let   $ \left\langle 
\mathcal{E}_\alpha:\ 
\alpha< \kappa  \right\rangle $ be as in Fact \ref{fact: elementary chain} where $ x=M $. By Lemma 
\ref{lem: Laviolette key}, for every $ 
D\in \mathcal{C}^{*} $ and for the smallest ordinal $ \alpha+1 $ for which $ D\cap\mathcal{E}_{\alpha+1} 
\neq \emptyset $, we have $ D\subseteq\mathcal{E}_{\alpha+1}\setminus \mathcal{E}_{\alpha} $ if $ \left|D\right|\leq 
\lambda $ and $ \left|D\cap 
(\mathcal{E}_{\alpha+1}\setminus \mathcal{E}_{\alpha}) \right|=\lambda $ if $ \left|D\right|>\lambda $.  We define 
$ 
F_\alpha:=E 
\cap (\mathcal{E}_{\alpha+1}\setminus \mathcal{E}_\alpha) $ for $ \alpha<\lambda $. Then the 
partition $ E=\bigcup_{\alpha<\lambda}F_\alpha $ satisfies \ref{item: faithful} and  \ref{item: is somewhere}, moreover $ 
\left|F_\alpha\right|<\kappa $ for each $ \alpha<\kappa $. Apply the induction hypothesis to $ M | 
F_\alpha $ and let $ 
F_\alpha= \bigcup_{i \in I_\alpha}F_{\alpha, i} $ be the 
resulting partition. We claim that $ E= \bigcup_{\alpha<\lambda, i\in I_\alpha}F_{\alpha, i} $ is a $ \lambda $-cocircuit-faithful partition of $ 
M $. Property \ref{item: small} is clear by the induction hypothesis. If $ D\in \mathcal{C}^{*}(M | 
F_{\alpha, i}) $ with $ 
\left|D\right|<\kappa $, then by induction $ D\in \mathcal{C}^{*}(M | F_{\alpha}) $. But then  $ D\in 
\mathcal{C}^{*} $ because partition $ E=\bigcup_{\alpha<\lambda}F_\alpha $ satisfies \ref{item: faithful}. Therefore 
\ref{item: faithful} holds 
true. Finally, let $ D\in \mathcal{C}^{*} $ with $ \left|D\right|\leq \kappa $. Then 
there is an $ \alpha $ with $ D\in \mathcal{C}^{*}(M | F_\alpha) $ because partition $ 
E=\bigcup_{\alpha<\lambda}F_\alpha 
$ satisfies \ref{item: is somewhere}. But then by induction there is an $ i\in I_\alpha $ 
such that $ D\in \mathcal{C}^{*}(M | F_{\alpha,i}) $. Thus \ref{item: is somewhere} holds as well.
\end{proof}
\section{Farkas lemma in finitary oriented matroids}\label{sec: farkas}
First, we give a simple example due to Kaspar showing that the Farkas lemma (Fact \ref{fact: Farkas lemma}) does not 
extend to finitary oriented matroids. Then we prove 
that it does extend for 
binary ones. 
\begin{claim}\label{claim: Farkas fail}
There exists a finitary oriented matroid $ \vec{M}=(E, \mathcal{O}) $ that has neither a positive circuit nor a positive 
cocircuit.
\end{claim}
\begin{proof}
Let $ \vec{M}:=(\mathbb{Z}, \mathcal{O}) $ where $ \underline{\mathcal{O}} $ consists of the $ 3 $-element subsets of $ 
\mathbb{Z} $ and for 
$ 
i<j<k $ the signed set $ \{ \left\langle i, 1  \right\rangle, \left\langle j, -1  \right\rangle, \left\langle k, 1  \right\rangle \} $ and 
its negation are in $ 
\mathcal{O} $. Then $ \underline{\mathcal{O}^{*}}=\{ \mathbb{Z}\setminus \{ n \}:\ n \in \mathbb{Z} \} $ and  $ 
\mathcal{O}^{*} $ consists of  
the signed set $ \{ \left\langle m, 1  \right\rangle :\ m<n  \}\cup \{ \left\langle m,-1  \right\rangle:\ n<m \} $ and its negation 
for each $ n\in 
\mathbb{Z} $. It is straightforward to verify that the orthogonality axiom (\ref{ax: orthog}) holds, thus $ \mathcal{O} $ and $ 
\mathcal{O}^{*} $ are indeed the circuit signature and cocircuit signature of a matroid respectively.
\end{proof}
\Farinf

\begin{proof}
We cannot have both by Fact \ref{fact: bin orthog}. Suppose for a contradiction that there is neither a positive circuit 
 nor a positive cocircuit through $ e_0 \in E $ in $ \vec{M} $. By Fact \ref{fact: Farkas lemma}  and by the non-existence of 
 positive circuits 
 through $ 
 e_0 $, 
 we know 
 that for every finite $ 
 E'\subseteq E $ containing $ e_0 $, we can find a positive cocircuit of $ \vec{M} |  E' $ through $ e_0 
 $. Thus 
 by compactness there exists a set $ F\subseteq E $ containing $ e_0 $ such that   $ 
 \sum_{e\in F \cap \underline{C}}C(e) =0$ for every $ C\in \mathcal{O} $. 
 By Zorn's lemma we can 
 assume that $ F $ is minimal with respect to these properties. Clearly, $ 
 \left|F\cap 
 \underline{C}\right| \neq 1$ for each $ C\in 
 \mathcal{O} $. In particular, there is no $ C\in 
\mathcal{O} $ with $F\cap 
 \underline{C}=\left\lbrace e_0\right\rbrace $. Thus by Fact \ref{fact: football field}
  there is a $ D\in 
 \mathcal{O}^{*} $ with $ e_0\in 
 \underline{D} \subseteq 
 F $. By negating $ D $ if necessary we can assume $ 
 D(e_0)=1 $.  By the indirect assumption $ D $ is not a positive cocircuit, thus there is an $ e_1\in \underline{D} $ with $ 
 D(e_1)=-1 $. Let $\mymathbb{1}_F $ be the characteristic function of $ F $ (defined on $ E $) and let $ 
 \widehat{D} $ be the extensions of $ D $  to the domain $ E $ with $ 0 $ 
 values. We define $ f:= \widehat{D} + \mymathbb{1}_F  $. 
 Then 
 \begin{enumerate}[label=(\roman*)]
 \item\label{item: f noon-negative} $ f \geq 0 $,
 \item\label{item: f e0 is 2} $ f(e_0)=2 $,
 \item\label{item: f is orthog} $ \sum_{e\in \underline{C}}f(e)C(e)=0 $ for every $ C\in \mathcal{O} $.
 \end{enumerate}
 For  the support $ S:=\{ 
 e\in E:\ f(e)\neq 0 \} $  of $ f $ 
 we have
 $S\subseteq F\setminus \{ e_1 \}\subsetneq F $. Thus in order to get a contradiction to the minimality of $ F $, 
we will use compactness to construct a $ G $ with  $ e_0 \in G\subseteq 
 S $ such that $ \sum_{e\in G \cap \underline{C}}C(e)=0 $ for every $ C\in \mathcal{O} $. 
 
  Let  $ S'\subseteq S $ be any finite set containing $ e_0 $. 
 If $ (\vec{M}.S) | S'$ has a positive circuit $ C $ through $ e_0 $, then by Fact \ref{fact: oriented 
 minor} there is a $ C'\in 
 \mathcal{O} $ 
 with $ C' \!\!\upharpoonright\!\! S = C  $. But then  by applying \ref{item: f noon-negative} and \ref{item: f e0 is 2} together 
 with the positivity of 
 $ C $ we obtain
  \[ \sum_{e\in \underline{C'}}C'(e) f(e) =\sum_{e\in \underline{C}}C(e) f(e) \geq f(e_0) C(e_0) =2, \]  which contradicts 
  \ref{item: f is orthog}. 
 Hence there is no 
 such a circuit and therefore  $ (M.S) 
 | S'$ has a 
 positive cocircuit $ D' $
 through $ e_0 $ by Fact \ref{fact: Farkas lemma}. Then $ e_0\in \underline{D'}  \subseteq S  $ 
 and we claim that for every $ C\in \mathcal{O} $ with $ \underline{C}\cap S\subseteq S' $ we have $ 
 \sum_{e\in \underline{D'}\cap \underline{C}} C(e)=0 $. Indeed, $ D' \in \mathcal{O}^{*}(M.S 
 |  S') $ implies that there is a $ D''\in \mathcal{O}^{*} $ 
with $  D''\!\!\upharpoonright\!\!  S'=D' $ and $ \underline{D''}\subseteq S $ (see Fact 
\ref{fact: oriented minor}), thus $ \sum_{e\in \underline{D'}\cap \underline{C}} C(e)=\sum_{e\in \underline{D''}\cap 
\underline{C} } C(e)D''(e)=0 $ because $ \underline{D''}\cap \underline{C}=\underline{D'}\cap \underline{C}\subseteq 
S'$. 
 But then, since the finite 
 set $ S'\subseteq S $ was arbitrary,  it follows by compactness that the desired $ G$ exists, which concludes the proof.
\end{proof}

\section{Partitioning oriented matroids into directed circuits}\label{sec: oriented}
In an oriented matroid $ \vec{M}=(E,\mathcal{O}) $  a $ D\in \mathcal{O}^{*} $ is called \emph{balanced} if $ \left|D^{+}\right| 
=\left|D^{-}\right|$. We 
call $ \vec{M} $ itself \emph{balanced} if every $ D\in \mathcal{O}^{*} $ is balanced. 
\orparti

\begin{proof}
Let $ \vec{M}=(E,\mathcal{O}) $ be a finitary oriented matroid. Suppose that  $ \mathcal{A} $ is a set of directed circuits, 
say positive 
ones,  such that $\underline{ \mathcal{A}} $ is a partition of $ E $. Then, for any $ 
D\in \mathcal{O}^{*} $ and $ s\in \{ +,- \} $, we have $ 
D^{s}=\dot{\bigcup} \{\underline{C}\cap D^{s}:\   C\in \mathcal{A}\}  $. It follows from Fact \ref{fact: bin orthog}, that 
we 
must have $ \left|\underline{C}\cap D^{+}\right|=\left|\underline{C}\cap D^{-}\right| $ for every positive circuit $ C $, 
therefore 
\[ \left|D^{+}\right|=\sum_{C\in \mathcal{A}}\left|\underline{C}\cap D^{+}\right|=\sum_{C\in 
\mathcal{A}}\left|\underline{C}\cap D^{-}\right|=\left|D^{-}\right|. \]  

We turn to the 
non-trivial direction. Let a regular uncountable cardinal $ \varTheta $ be fixed. We need a couple of lemmas:
\begin{lemma}\label{lem: remove balanced}
If $ \mathcal{E}\prec H(\varTheta) $ with $ \left|\mathcal{E}\right|\subseteq 
\mathcal{E} $, then for every finitary balanced oriented matroid $ \vec{M}\in \mathcal{E} $,  the matroid $ 
\vec{M}\setminus 
\mathcal{E}$ is also balanced.
\end{lemma}
\begin{proof}
Let $ D_0\in \mathcal{O}^{*}(\vec{M}\setminus \mathcal{E}) $ be given. By Fact \ref{fact: oriented minor} there is a $ 
D_1\in 
\mathcal{O}^{*}(\vec{M}) $ such that $ D_0= D_1 \!\!\upharpoonright\!\! (E\setminus \mathcal{E}) $. If $ 
\underline{D_1}\cap 
\mathcal{E}=\emptyset $, then $ D_0=D_1 \in  
\mathcal{O}^{*}(\vec{M})$ and 
we are done because $ D_1 $ is balanced by assumption. Suppose that $ \underline{D_1}\cap \mathcal{E}\neq \emptyset $. 
Then it follows by applying 
Lemma 
\ref{lem: Laviolette key} with $ \underline{D_1} $ and $ M $ that we must have
$ \left|\underline{D_1}\right|>\left|\mathcal{E}\right| $. Since $ \left|D_1\setminus D_0\right|<\left|\mathcal{E}\right| $, we 
conclude that $ 
\left|D_1^{s}\right|=\left|D_0^{s}\right| $ for $ s\in \{ +,- \} $, therefore
\[ \left|D_0^{+}\right|=\left|D_1^{+}\right|=\left|D_1^{-}\right|=\left|D_0^{-}\right|. \]
\end{proof}
\begin{lemma}\label{lem: restrict balanced}
If $ \mathcal{E}\prec H(\varTheta) $ with $ \left|\mathcal{E}\right|\subseteq \mathcal{E} $, then for every finitary binary 
balanced  oriented matroid 
$ \vec{M}=(E,\mathcal{O})\in \mathcal{E} $, the 
matroid $ 
\vec{M}| \mathcal{E} 
$ is also balanced.
\end{lemma}
\begin{proof}
Let $ D\in 
\mathcal{O}^{*}(\vec{M}|  
\mathcal{E}) $ be arbitrary. By symmetry we may 
assume that $\left| D^{+}\right| \geq \left| D^{-}\right|  $. We need to show that $\left| D^{+}\right| \leq \left| D^{-}\right|  $. 
Clearly, $ D^{+}\neq \emptyset  $ because $ D\neq \emptyset  $. Fix 
an $ e_0\in D^{+} $  and let $ 
\vec{N}:=\vec{M}/(\mathcal{E}\setminus \underline{D}) \setminus (\underline{D} \setminus \{ e_0\}) $. 
We apply Lemma \ref{lem: Farkas property} 
with $ e_0 $ and $\vec{N} $. Suppose first that there is a positive  $ D_0 \in \mathcal{O}^{*}(\vec{N}) $ through $ e_0 $. 
Then there is a $ D_1\in \mathcal{O}^{*}(\vec{M}) $ with $ D_1\!\!\upharpoonright\!\! ((E\setminus \underline{D})\cup 
\{ e_0 \})=D_0 $ and $  \underline{D_1}\cap \mathcal{E}\subseteq 
\underline{D} $  (see Fact 
\ref{fact: oriented minor}).  We must have $ \underline{D_1}\cap \mathcal{E}= \underline{D} $ because $ \underline{D_1}\cap \mathcal{E} 
$ includes a  cocircuit of $M  |  \mathcal{E} $ by Fact \ref{fact: contract cocircuit} but this cannot 
be a proper subset 
of another cocircuit $ \underline{D} $. Then $ 
D_1 \!\!\upharpoonright\!\!  \mathcal{E}=\pm 
D $ and by $ D_1(e_0)=D(e_0)=1 $ we 
conclude that $ D_1 \!\!\upharpoonright\!\!  \mathcal{E}=D $. But then by the positivity of $ D_1 \setminus D $ and by the 
assumption that $ 
\vec{M} $ 
is balanced: 
\[\left|D^{+}\right|\leq \left|D_1^{+}\right|=\left|D_1^{-}\right|=\left|D^{-}\right|. \]

Assume now that there is a positive $ C_0\in \mathcal{O}(\vec{N}) $ through $ e_0 
$. Then there 
is a $C_1\in  \mathcal{O}(\vec{M}) $ with $ C_1\!\!\upharpoonright\!\! (E\setminus (\mathcal{E}\setminus \underline{D})) 
=C_0 $. Let $ F:=C_1\cap \mathcal{E}  $ 
and let  
$\mathcal{X}\in \mathcal{E} $ 
be the set of those  positive signed sets $ X $ with $\underline{X}\subseteq E \setminus \underline{F} $
for which 
$ F \cup X\in \mathcal{O}(\vec{M}) $. Then $ \underline{C_1}\setminus 
\mathcal{E} \in\underline{\mathcal{X}} $ is disjoint from $\mathcal{E}$  and $\underline{C_1}\setminus 
\mathcal{E}\neq \emptyset$ because $e_0$ cannot be an $M$-loop. It follows by Fact \ref{fact: lot of disjoint} that one can find  
$\left|\mathcal{E}\right| 
$ many pairwise disjoint elements in
$\underline{\mathcal{X}}\cap\mathcal{E}$. To prove $\left| D^{+}\right| \leq \left| D^{-}\right|  $, it suffices to show 
that for each  $ X\in \mathcal{X}\cap \mathcal{E} $, there is an $ e\in \underline{D} \cap \underline{X} $ with $ D(e)=-1 $. 
Suppose that $ X $ 
is a 
counterexample, and let $ D'\in \mathcal{O}^{*}(\vec{M}) $ with $D'\!\!\upharpoonright\!\! \mathcal{E} = D$. Note that 
$ \underline{D'}\cap (\underline{F}\cup \underline{X})=\underline{D}\cap (\underline{F}\cup \underline{X})$ because $ 
\underline{F}\cup \underline{X} \subseteq \mathcal{E} $ and recall that $ \underline{F}\cap 
\underline{D}=\{ e_0 \} $ by construction. Then by using the indirect 
assumption and the positivity 
of $ X  $:
\[ 1=D(e_0)F(e_0)\leq 
\sum_{e\in \underline{D}\cap (\underline{X} \cup \underline{F})}D(e)(F\cup X)(e)=\sum_{e\in \underline{D'}\cap 
(\underline{X} \cup \underline{F})}D'(e)(F\cup X)(e). \]
 But $ D' \in \mathcal{O}^{*}(\vec{M}) $ and $ F\cup X \in \mathcal{O}(\vec{M}) $, thus this contradicts Fact \ref{fact: bin orthog}.
\end{proof}

 We prove by induction on $ 
\kappa:=\left|E\right| $ that every finitary binary balanced oriented matroid $ \vec{M}=(E,\mathcal{O}) \in H(\varTheta) $ 
admits a partition into 
positive circuits.
\begin{observation}\label{obs: balanced e in pos circ}
In a finitary binary balanced oriented matroid every edge is contained in a positive circuit.
\end{observation}
\begin{proof}
This follows directly from Lemma \ref{lem: Farkas property} because being balanced ensures the non-existence of positive 
cocircuits.
\end{proof} 
\begin{observation}\label{obs: remains balanced}
If $ \mathcal{A} $ is a finite set of pairwise disjoint  positive circuits of $ \vec{M} $, then $ \vec{M}\setminus 
\bigcup\underline{\mathcal{A}} $ 
is balanced.
\end{observation}
\begin{proof}
Let $ D\in \mathcal{O}^{*}(\vec{M}\setminus \bigcup\underline{\mathcal{A}}) $. Pick a  $ D'\in \mathcal{O}^{*} $ with 
$ D'\!\!\upharpoonright\!\!  (E\setminus \bigcup\underline{\mathcal{A}})=D $. Since $ \vec{M} $ is binary and each $ C\in 
\mathcal{A} $ is positive, Fact \ref{fact: bin orthog} ensures that $ D' $ has the same number of positive and negative edges 
in $ 
\bigcup\underline{\mathcal{A}} $. But then so does $ D $ because $ D' $ is balanced by assumption.
\end{proof}
\noindent If  $ \kappa \leq \aleph_0 $, then the desired 
partition can be obtained by a straightforward recursion via Observations \ref{obs: balanced e in pos 
circ} and \ref{obs: remains balanced}. Indeed,  we fix a $\kappa$-type well-order of $ E $  and in step $ n $ apply 
Observation \ref{obs: balanced e in pos circ} to $ \vec{M} \setminus \bigcup_{i<n}\underline{C_i} $ (which is balanced by 
Observations 
\ref{obs: remains balanced}) to find a positive circuit $ C_n $  disjoint from $ C_0,\dots, C_{n-1} $ through 
the smallest element of $ E\setminus \bigcup_{i<n}C_i $ . 

Suppose $ \kappa >\aleph_0 $. Let $ \left\langle \mathcal{E}_\alpha:\ \alpha<\kappa  \right\rangle $ be as in Fact \ref{fact: 
elementary chain} with 
$ x=\vec{M} $ and $ \lambda=\aleph_0 $. Then, for each $ \alpha<\kappa $: $ \vec{M}\setminus \mathcal{E}_\alpha  $ is 
balanced by Lemma 
\ref{lem: remove balanced}. Since $ \vec{M}, \mathcal{E}_\alpha\in \mathcal{E}_{\alpha+1} $, we have $ 
\vec{M}\setminus 
\mathcal{E}_\alpha \in \mathcal{E}_{\alpha+1} $, thus by Lemma \ref{lem: restrict balanced} we conclude that 
$ (\vec{M}\setminus \mathcal{E}_\alpha) | \mathcal{E}_{\alpha+1} $ is balanced. By the induction 
hypotheses, there is a 
partition $ \mathcal{A}_\alpha $  of $ (\vec{M}\setminus \mathcal{E}_\alpha) | 
\mathcal{E}_{\alpha+1} $ into positive 
circuits. Finally, $ \bigcup_{\alpha<\kappa}\mathcal{A}_\alpha $ is a partition of $ \vec{M} $ into positive circuits.
\end{proof}

\section{Open questions}\label{sec: open questions}
A concept of infinite matroids, which is more general than finitary matroids, was introduced by Higgs and was investigated 
by Higgs and Oxley in the
1960s and 1970s  (see \cite{higgs1969matroids, oxley1978infinite}). The investigation of this concept gained a new 
momentum half a century later
when Bruhn et al.  rediscovered it independently and provided five sets of crypcryptomorphic 
axiomatisations (see \cite{bruhn2013axioms}). The axiomatisation in terms of independent sets reads as follows: 
$ M=(E, \mathcal{I}) $ is a matroid if $ \mathcal{I}\subseteq \mathcal{P}(E) $ such that
\begin{enumerate}
\item $ \emptyset \in \mathcal{I} $,
\item $ \mathcal{I} $ is downward closed,
\item whenever $ B $ is a maximal and $ I $ is a non-maximal element of $ \mathcal{I} $, then there is an $ e \in B 
\setminus I $ 
such that $ I \cup \{ e \}\in \mathcal{I} $,
\item for every $ X\subseteq E $, every $ I\in \mathcal{P}(X) \cap \mathcal{I} $ can be extended to a maximal element of $ 
\mathcal{P}(X) \cap \mathcal{I} $.
\end{enumerate}

Finitary matroids are exactly those matroids in which all circuits are finite. Matroids that are not finitary are called 
\emph{infinitary}. Relatively little is known about perfect matchings (i.e. circuit-partitions) of infinitary matroids. Next we formulate a question 
ispired by Theorem \ref{thm: fme match}.

 For an infinite cardinal $ \kappa $, we call a (not 
necessarily finitary) matroid   $ \kappa $-\emph{matching 
 extendable}  if for every matching (i.e. set of pairwise disjoint circuits) $ \mathcal{A} $  of size less than $ \kappa $ and for 
 every edge  $ e \in E\setminus \bigcup \mathcal{A} $ there is a circuit $ C $ through $ e $ with $ C \cap \bigcup 
 \mathcal{A} =\emptyset  $.

 \begin{question}
 Is it true that every $ \kappa $-matching extendable matroid whose circuits are smaller than $ \kappa $ admits a perfect 
 matching? 
 \end{question}
\noindent The positive answer for $ \kappa=\omega $ is exactly Theorem \ref{thm: fme match}.
 
The definition of the oriantability of infinitary matroids is somewhat problematic in general (see \cite{hoch2015orthogax}) 
but works well for a certain subclass. A matroid is called \emph{tame} if the intersection of any circuit with any cocircuit is 
finite. As in the finitary case, a tame matroid is orientable if there is a signature of its circuits and cocircuits 
such that they satisfy (\ref{ax: orthog}). We do not know whether ``finitary'' can be replaced by the weaker assumption 
``tame'' in Lemma \ref{lem: Farkas property}, i.e. the following question is open:

\begin{question}
Is it true that in every tame binary oriented matroid $ \vec{M}=(E, \mathcal{O}) $,  for every $ e\in E $, there is either a 
positive 
circuit 
or a positive cocircuit through $ e $?
\end{question}

\printbibliography
\end{document}